\documentclass[11pt,a4paper]{amsart}
\usepackage{fullpage}
\usepackage{amssymb,amscd}
\usepackage[pdfencoding=auto]{hyperref}
\usepackage{adjustbox}
\usepackage[all]{xy}
\usepackage{textcomp}
\usepackage{tikz}
\usepackage{color}
\usetikzlibrary{shapes,arrows}
\usepackage{cleveref}

\newtheorem{thm}{Theorem}[section]
\newtheorem{lem}[thm]{Lemma}
\newtheorem{prop}[thm]{Proposition}

\theoremstyle{definition}
\newtheorem{defn}[thm]{Definition}
\newtheorem{ex}[thm]{Example}

\theoremstyle{remark}
\newtheorem{rem}[thm]{Remark}
\newtheorem{notation}[thm]{Notation}

\newcommand{\F}{\mathbb{F}}

\def\[#1\]{\begin{align*}#1\end{align*}}

\def\Tor{{\sf Tor}}

\def\Z{\mathbb{Z}}

\def\F{\mathbb{F}}

\def\geq{\geqslant}
\def\call{\mathcal{L}}
\def\ra{\rightarrow}
\def\HH{{\sf HH}}
\def\THH{{\sf THH}}

\def\Shukla{{\sf Shukla}}

\def\ie{\emph{i.e.}}

\def\id{\mathrm{id}}

\newcommand{\botimes}{\mathbin{\bar{\otimes}}}
\newcommand{\twedge}{\mathbin{\tilde{\wedge}}}
\renewcommand{\L}{\mathcal{L}}
\numberwithin{equation}{thm}
\begin{document}
\title{Relative Loday constructions and applications to higher
  $\THH$-calculations}
\author{Gemma Halliwell}
\address{School of Mathematics and Statistics, University of
  Sheffield, Hicks Building, Hounsfield Road, S3 7RH Sheffield, UK}
\email{glhalliwell1@sheffield.ac.uk}

\author{Eva H\"oning}
\address{Institut Galil\'ee, Universit\'e Paris 13,
99, avenue Jean-Baptiste Cl\'ement, 93430 Villetaneuse, France}
\email{hoening@math.univ-paris13.fr}

\author{Ayelet Lindenstrauss}
\address{Mathematics Department, Indiana University, Bloomington, IN 47405, USA}
\email{alindens@indiana.edu}

\author{Birgit Richter}
\address{Fachbereich Mathematik der Universit\"at Hamburg,
Bundesstra{\ss}e 55, 20146 Hamburg, Germany}
\email{birgit.richter@uni-hamburg.de}

\author{Inna Zakharevich}
\address{
Department of Mathematics, The University of Chicago, 5734 S
University Ave,
Chicago, IL 60637, USA}
\email{zakh@math.uchicago.edu}

\date{\today}
\begin{abstract}
We define a relative version of the Loday construction for a sequence
of commutative $S$-algebras $A \ra B \ra C$ and a pointed simplicial
subset $Y \subset X$. We use this to construct several spectral
sequences for the calculation of higher topological Hochschild
homology and apply those for calculations in some examples that could
not be treated before.

\end{abstract}
\thanks{We thank the Banff
International Research Station for their support and hospitality.  Our
warm thanks go the organizers of the BIRS workshop
\emph{WIT II: Women in Topology 2016}, Maria Basterra, Kristine Bauer,
Kathryn Hess and Brenda Johnson.}
\maketitle

\section{Introduction}
Considering relative versions of (co)homology theories is crucial for
obtaining calculational and structural results. We work in the setting
of commutative $S$-algebras (see \cite{EKMM}) and these can be tensored with
simplicial sets or topological spaces. Important invariants of a
commutative $S$-algebra $A$ are the homotopy groups of $A \otimes X$
for suitable spaces or simplicial sets $X$. We call this the
Loday construction of $A$ with respect to $X$ and denote it by
$\L_X(A)$. 

Important special
cases are the higher order topological Hochschild homology of $A$, 
$$\THH^{[n]}_*(A) = \pi_*(A \otimes S^n), $$ 
where $S^n$ is the $n$-sphere. For $n=1$ this reduces to ordinary
topological Hochschild homology of $A$, $\THH(A)$, which receives a
trace map from the algebraic K-theory of $A$, $K(A)$, and can be used via the
construction of topological cyclic homology to obtain an approximation
of $K(A)$. Torus
homology of $A$, $\pi_*(A\otimes T^n)$ for the $n$-torus  
$T^n$, receives an $n$-fold iterated trace map from the
iterated algebraic K-theory of $A$. Ongoing work by Ausoni and Dundas
uses torus homology in order to make progress on the so-called red-shift
conjecture \cite{ausoni-rognes}. 

One strength of the construction of $A \otimes X$ is that it is
functorial in both $X$ and $A$, which allows us to study the homotopy
type of $A \otimes X$ by iteratively constructing $X$ out of
smaller spaces.  This iterative method is for instance heavily used in
Veen's work \cite{V} and in \cite{DLR}, \cite{BLPRZ}. As spheres are
the building blocks of CW complexes, 
the calculation of $\THH^{[n]}_*(A)$ is crucial for understanding
$\pi_*( \L_X(A))$ for CW-complexes $X$. The aim of this paper is to
develop new tools for the calculation of higher order topological
Hochschild homology by using the extra flexibility that is gained by a
relative approach. 

For a sequence of morphisms in the
category of commutative $S$-algebras $A \ra B \ra C$ and for a pair of
pointed simplicial
sets $(X, Y)$ we consider a relative version of the Loday
construction, $\L_{(X,Y)}(A, B; C)$: This relative version places $C$
over the basepoint,
$B$ over all points of $Y$ that are not the basepoint and $A$ over the
complement of $Y$ in $X$.

We define this relative Loday construction and show some of its
properties in Section \ref{sec:relloday}. Section \ref{sec:spsq}
exploits this relative structure and other geometric observations to
establish several weak equivalences (juggling formulae) that compare
higher $\THH$ groups
with respect to the sphere spectrum as the ground ring and those with respect to
other
commutative $S$-algebra spectra as ground rings, and compare higher
$\THH$-groups $\THH^{[m]}$ for $m = n$ and $n-1$. We use these
juggling formulae to
construct several spectral sequences for the calculation of higher
topological Hochschild homology and apply these spectral sequences to
obtain calculations for (higher)
topological Hochschild homology that were not known before. In Section
\ref{sec:ex} we determine for instance higher order relative $\THH$ of
certain Thom spectra and the higher order Shukla homology of
$\F_p$ with respect to pointed commutative $\F_p$-monoid algebras. We
gain additive results about $\THH^E(H\F_p)$ and $\THH^{[2]}(E; H\F_p)$
for $E = ko, ku, \mathit{tmf}$ at $p=2$ and $E=\ell$ the Adams
summand for $p$ an odd prime. Furthermore, we show
a splitting result for higher $\THH$ of the form $\THH^{[n], Hk}(HA)$,
where $k$ is any commutative ring and $A$ is any commutative
$k$-algebra.

In the following we work in the setting of \cite{EKMM} and we use the 
model structure on commutative $S$-algebras from
\cite[chapter VII]{EKMM}. Let $A$ be a commutative $S$-algebra. As the
category of commutative $A$-algebras is equivalent to the category of
commutative $S$-algebras under $A$, we obtain a model category
structure on the category of commutative $A$-algebras. In particular,
a commutative $A$-algebra $B$ is cofibrant if its unit map $A \ra B$
is a cofibration of commutative $S$-algebras.

\section{The relative Loday construction}
\label{sec:relloday}
Higher topological Hochschild homology of a commutative ring spectrum
$A$ is a special case  of the Loday
construction, or the internal
tensor product, which sends $A$ and a simplicial set $X$ to a
commutative simplicial ring spectrum $\call_X(A) = X \otimes A$, which
is a commutative augmented $A$-algebra.
This is a ring spectrum version of the Loday construction defined by
Pirashvili in \cite{P} for commutative rings, which sends a
commutative ring $R$ and a simplicial set $X$ to the
commutative augmented simplicial $R$-algebra $X \otimes R$. 

For a cofibrant commutative $S$-algebra  $A$  
this construction is homotopy invariant as a functor of $X$, that is: if
one works with homotopy equivalent simplicial sets, we get homotopy
equivalent augmented simplicial commutative $A$-algebras; 
in particular, this is true if one
works with two simplicial models for the same space.

Let $X$ be a pointed simplicial set. Since all boundary maps in a
pointed simplicial set send the basepoint to the basepoint, given an
$A$-module $C$ (and in
particular, a commutative $A$-algebra $C$) we can also study Loday
constructions with coefficients $\call_X(A; C)$ which replaces the
copy of $A$ over the base point by a copy of $C$.
We now define a relative version of this.

\begin{defn} \label{def:loday}
Let $A$ be a commutative $S$-algebra, $B$ a commutative $A$-algebra,
and $C$ a commutative $B$-algebra, with maps
$A\xrightarrow{f} B \xrightarrow{g} C$. Let $X$ be a pointed
simplicial set and $Y$ be a pointed simplicial subset. Then we can define
\[(\call_{(X,Y)}(A,B;C))_n := \left(\bigwedge_{X_n\setminus Y_n} A
\right) \wedge \left( \bigwedge_{Y_n \setminus \ast} B \right) \wedge
C. \]
We call this the \emph{$n$th simplicial degree of the relative Loday
construction of $A$ and $B$ with coefficients in $C$ on $(X,Y)$}.

The smash product here is usually taken over the sphere spectrum, but
can be done over any commutative ring spectrum, $k$. In this case, we
will add a superscript to the notation, $$\call^k_{(X,Y)}(A,B;C).$$
If $B=C$ then the simplicial subset $Y$ of $X$ does not have to be
pointed.

The structure maps of this construction use the fact that the smash
product is the coproduct  
in the category of commutative $S$-algebras and they  are given as follows:
Let $\varphi \in \Delta([m], [n])$ and let $\varphi^*$ denote the
induced map on $X$ and $Y$:
$$ \varphi^* \colon Y_n \ra Y_m, \quad \varphi^* \colon X_n \ra X_m.$$
Note that $X_n \setminus Y_n$ is not a subcomplex of $X_n$, so
$\varphi^*$ might send elements in here to $Y_m$.
We get an induced map
$$ \varphi^* \colon \left(\bigwedge_{X_n\setminus Y_n} A
\right) \wedge \left( \bigwedge_{Y_n \setminus \ast} B \right) \wedge
C \ra  \left(\bigwedge_{X_m\setminus Y_m} A
\right) \wedge \left( \bigwedge_{Y_m \setminus \ast} B \right) \wedge
C$$
by  first defining a map from
\begin{equation} \label{eq:smash}
\left(\bigwedge_{X_n\setminus Y_n} A
\right) \wedge \left( \bigwedge_{Y_n \setminus \ast} B \right) \wedge
C
\end{equation}
to a smash product over $X_n$ where some of the smash factors $A, B$ in
\eqref{eq:smash} are sent to $B$ or $C$:
\begin{itemize}
\item[]
\item
If for $x \in X_n\setminus Y_n$ the image $\varphi^*(x)$ is in
$Y_m \setminus *$,
then we apply the map $f$ on the corresponding smash factor $A$.
\item
If $x \in X_n\setminus Y_n$ is sent to the basepoint in $Y_m$ under
$\varphi^*$, then we use the composite $g \circ f$ on the
corresponding factor $A$.
\item
If for $x \in X_n\setminus Y_n$, the element $\varphi^*(x)$ is in
$X_m\setminus Y_m$, then we don't do anything, \ie, we apply the
identity map on the corresponding factor $A$.
\item
Similarly, if a $y \in Y_n \setminus *$ is sent to the basepoint, then
we apply $g$ to the corresponding $B$-factor.
\item
If $y \in Y_n \setminus *$ is sent to $Y_m \setminus *$, then we apply
the identity map to the corresponding $B$-factor.
\item
By assumption, the basepoint is sent to the basepoint, so the
$C$-factors are not involved in this process.
\end{itemize}
We now use the map $\varphi$ to obtain a map to
$\left(\bigwedge_{X_m\setminus Y_m} A
\right) \wedge \left( \bigwedge_{Y_m \setminus \ast} B \right) \wedge
C$:
\begin{itemize}
\item[]
\item
If an $x \in X_m \setminus Y_m$ has multiple preimages under
$\varphi^*$, then we use the multiplication on $A$ on the
corresponding smash factors.
\item
If a $y \in Y_m \setminus *$ has multiple preimages under
$\varphi^*$, then we use the multiplication on $B$ on the
corresponding smash factors.
\item
If the basepoint in $Y_m$ has multiple preimages, then we apply the
multiplication in $C$ on the corresponding factors.
\item
If some $x \in X_m \setminus Y_m$ is not in the image of $\varphi^*$,
then we insert the unit map of $A$ in the corresponding spots; similarly if
$y \in Y_m \setminus *$ has an empty preimage we use the unit map of
$B$. The basepoint always has a non-empty preimage.
\end{itemize}
As the multiplication maps on $A$, $B$ and $C$ are associative and
commutative and as the maps $f$ and $g$ are morphisms of commutative
$S$-algebras, this gives the relative Loday construction the structure
of a simplicial spectrum.
\end{defn}

\begin{lem}
The relative Loday construction, $\call_{(X,Y)}(A,B;C)_\bullet$, is
 a simplicial augmented commutative $C$-algebra spectrum.
\end{lem}
\begin{proof}
The multiplication
$$ \call_{(X,Y)}(A,B;C)_n \wedge_C \call_{(X,Y)}(A,B;C)_n \ra
\call_{(X,Y)}(A,B;C)_n$$
is defined coordinatewise and is therefore compatible with the
simplicial structure. Hence we obtain a simplicial commutative
augmented  $C$-algebra structure on
$\call_{(X,Y)}(A,B;C))_\bullet$.
\end{proof}
\begin{rem}
One can also define a version of the relative Loday construction if
$C$ is a $B$-module, rather than a commutative $B$-algebra. 
Then of course we cannot 
use the coproduct property of the smash product anymore, but we spelt out the 
structure maps explicitly, so that this generalization is not hard. 

In this
case $\call_{(X,Y)}(A,B;C))_\bullet$ is a simplicial $A$-module
spectrum and for the simplicial structure maps
one has to use the $B$-module structure of $C$. Note that the commutative
$S$-algebra $A$ acts on $C$ as well via the map $f \colon A \ra B$.
\end{rem}
\begin{ex}
As an explicit example of a pointed simplicial subcomplex we consider
$\partial \Delta_2 \subset \Delta_2$ whose basepoint $* \in \Delta([n],
[2])$ is the constant map with value $0$. Note that the number of
elements in $\Delta([n], [m])$ is $\binom{n+m+1}{n+1}$.

We describe the effect of the
maps $\varphi\colon [1] \ra [2]$, $\varphi(0)=0$, $\varphi(1) = 2$ and
$\psi \colon [2] \ra [1]$, $\psi(1)=\psi(0)=0$ and $\psi(2) = 1$.

\begin{center}
\setlength{\unitlength}{0.5cm}
\begin{picture}(13,2)(2,2)

\put(-2,0){$\varphi$}
\put(0,0){\vector(1,0){2}}
\put(0,1){\vector(2,1){2}}
\put(-0.5,-0.2){$0$}
\put(-0.5,0.8){$1$}
\put(2.5,-0.2){$0$,}
\put(2.5,0.8){$1$}
\put(2.5,1.8){$2$}
\put(6,0){$\psi$}
\put(8,0){\vector(1,0){2}}
\put(8,1){\vector(2,-1){2}}
\put(8,2){\vector(2,-1){2}}
\put(7.5,-0.2){$0$}
\put(7.5,0.8){$1$}
\put(10.5,-0.2){$0$}
\put(10.5,0.8){$1$}
\put(7.5,1.8){$2$}
\end{picture}

\end{center}
\vspace{1cm}
\bigskip

In $\L_{(\Delta_2, \partial \Delta_2)}(A, B; C)_2$ there is only one
copy of $A$ because
$$\Delta_2[2] \setminus \partial
\Delta_2([2]) = \Delta([2], [2]) \setminus \partial
\Delta_2([2]) = \{\id_{[2]}\}.$$
Thus
$$\L_{(\Delta_2, \partial \Delta_2)}(A, B; C)_2 = A \wedge \left(
  \bigwedge_{\partial 
\Delta_2([2]) \setminus \{\id_{[2]},*\}} B \right) \wedge C.$$
As $\partial \Delta_2[1] = \Delta_2[1]$ we get
$$\L_{(\Delta_2, \partial \Delta_2)}(A, B; C)_1 = \left(\bigwedge_{\partial
\Delta_2([1]) \setminus *} B\right) \wedge C.$$
The map $\psi^* \colon \Delta_2([1]) \ra \Delta_2([2])$ sends the six
elements of $\Delta_2([1])$ injectively to six elements in
$\Delta_2([2])$, so on the Loday construction we only use the unit
maps of $A$ and $B$ to fill
in the gaps. In particular, the identity of $[2]$ is not in the image
of $\psi$ and we get as $\psi^*$ on the Loday construction
$$ \xymatrix{{\left(\bigwedge_{\partial
\Delta_2([1]) \setminus *} B\right) \wedge C} \ar[r]^(0.3)\cong & S
\smash \wedge
\left(\bigwedge_{\Delta_2([2]) \setminus (\{\id_{[2]}\} \cup
  \text{im}(\psi^*))} S\right) \wedge \left(\bigwedge_{\text{im}(\psi^*)\setminus *}
B\right) \wedge C \ar[d]_{\eta_A \wedge
\left(\bigwedge _{\Delta_2([2]) \setminus (\{\id_{[2]}\} \cup
  \text{im}(\psi^*))} \eta_B\right) \wedge
\left(\bigwedge_{\text{im}(\psi^*)\setminus *} \id_B \right) \wedge \id_C} \\
 & A \wedge \left(\bigwedge_{\partial
\Delta_2([2]) \setminus \{\id_{[2]},*\}} B\right) \wedge C.}$$

In contrast to this, the map $\varphi^* \colon \Delta_2([2]) \ra
\Delta_2([1])$ is surjective. The preimage of the basepoint under
$\varphi^*$ is just the basepoint, so we get the identity on the
$C$-factor in $A \wedge \left(\bigwedge_{\partial
\Delta_2([2]) \setminus \{\id_{[2]},*\}} B\right) \wedge C$, but we have to
use the map $f\colon A \ra B$ and several instances of the
multiplication on $B$ to get to $\left(\bigwedge_{\partial
\Delta_2([1]) \setminus *} B\right) \wedge C$ because there is a fiber of
cardinality three and a fiber of cardinality two.
\end{ex}

\begin{ex} If in Definition \ref{def:loday} we take $A=B$, then
  $\L_{(X,Y)}(A,A;C) _\bullet = \L_X(A;C)_\bullet$. If in addition
  $A=C$ then we obtain $\L_{(X,Y)}(A,A;A) _\bullet= \L_X(A)_\bullet$.
\end{ex}
\begin{ex}
If we work relative to $A$, \ie, if we consider $\L^A_{(X,Y)}(A,B;C)_\bullet$,
then the $A$-factors disappear because we smash over $A$ and we get
$$ \L^A_{(X,Y)}(A,B;C)_\bullet \cong \L^A_{Y}(B;C)_\bullet. $$
\end{ex}

\begin{defn}
We define \emph{higher topological Hochschild homology of order $n$ of $A$ with
coefficients in $C$} by
$\THH^{[n],k}(A;C):= \call_{S^n}^k(A;C)_\bullet$. Here, $S^n$ is a pointed
simplicial model of the $n$-sphere.
\end{defn}

\begin{notation}
As above, if $k$ is the sphere spectrum
it is omitted from the notation. Similarly, if $n=1$ this may be
omitted also and written as $\THH^k(A;C)$. If $C=A$, we may write
$\THH^{[n],k}(A)$.
\end{notation}

\begin{prop} \label{prop:2.7}
\begin{enumerate}
\item[]
\item
For $A, B$, and as in Definition \ref{def:loday}, $C$ a commutative
$B$-algebra, $X$ a pointed
simplicial set and $Y$ a pointed simplicial subset, we get an
isomorphism of augmented simplicial commutative $C$-algebras
\begin{equation} \label{eqn:1}
\call_{(X,Y)}^k(A,B;C) _\bullet \cong \call_X^k(A;C) _\bullet
\wedge_{\call_Y^k(A;C)_\bullet } 
\call_Y^k(B;C)_\bullet.
\end{equation}
\item
For $X_0$ a common pointed simplicial subset of $X_1$ and $X_2$ and $Y_0$ a
common pointed simplicial subset of $Y_1$ and $Y_2$ so that $Y_i
\subseteq X_i$ for $i=1,2$ 
and $Y_0=X_0\cap Y_1\cap Y_2$, we have an
isomorphism of augmented simplicial $C$-algebras
\begin{equation} \label{eqn:2}
\call_{(X_1 \cup_{X_0} X_2, Y_1 \cup_{Y_0} Y_2)}^k(A,B;C) _\bullet\cong
\call_{(X_1,Y_1)}^k(A,B;C)_\bullet \wedge_{\call_{(X_0,Y_0)}^k(A,B;C)_\bullet}
\call_{(X_2,Y_2)}^k(A,B;C)_\bullet.
\end{equation}
\end{enumerate}
\end{prop}
If $C=B$, then in both statements we can work in the unpointed
setting.
\begin{proof}
For the claim in \eqref{eqn:1} we have a levelwise
isomorphism of simplicial spectra 
\begin{equation*}
(\call_{X}^k(A;C))_n \wedge_{(\call_Y^k(A;C))_n}
(\call_Y^k(B;C))_n \cong (\call_{(X,Y)}^k(A,B;C))_n
\end{equation*}
given by the identification of coequalizers
$$
\left(\bigwedge_{X_n \setminus Y_n} A \right) \wedge \left(
  \bigwedge_{Y_n \setminus \ast} A \right) \wedge C
\bigwedge_{\left( \left(\bigwedge_{Y_n \setminus \ast} A \right)
    \wedge C \right)} \left( \bigwedge_{Y_n \setminus \ast} B
\right) \wedge C 
\cong \left(\bigwedge_{X_n \setminus Y_n} A \right) \wedge \left(
  \bigwedge_{Y_n \setminus \ast} B\right) \wedge C. 
$$

Similarly, for \eqref{eqn:2} we have a levelwise isomorphism of
simplicial spectra
$$
(\call_{(X_1,Y_1)}^k(A,B;C))_n
\wedge_{(\call_{(X_0,Y_0)}^k(A,B;C))_n} (\call_{(X_2,Y_2)}^k(A,B;C))_n
\cong (\call_{(X_1 \cup_{X_0} X_2, Y_1 \cup_{Y_0} Y_2)}^k(A,B;C))_n.
$$
Here we use that tensoring a commutative $S$-algebra with a pointed
simplicial set is compatible with pushouts of simplicial sets, hence
we get
\begin{multline*}
\left(\bigwedge_{(X_1)_n\setminus (Y_1)_n} A \wedge
  \bigwedge_{(Y_1)_n \setminus \ast} B \wedge C \right)
\bigwedge_{\left( \bigwedge_{(X_0)_n\setminus (Y_0)_n} A \wedge
    \bigwedge_{(Y_0) _n \setminus \ast} B \wedge C \right)} \left(
  \bigwedge_{(X_2)_n \setminus (Y_2)_n} A \wedge \bigwedge_{(Y_2)_n
    \setminus \ast} B \wedge C \right) \\ \cong
\bigwedge_{(X_1 \cup_{X_0} X_2)_n \setminus (Y_1 \cup_{Y_0} Y_2)_n} A
\wedge \bigwedge_{y\in (Y_1 \cup_{Y_0} Y_2)_n \setminus \ast} B \wedge C.
\end{multline*} 
\end{proof}

Let $A$ be an augmented commutative $C$-algebra, \ie, in addition to
the map $g \circ f \colon A \ra C$ we have a map $\eta \colon C \ra A$,
such that $g \circ f \circ \eta = \id_C$. In that case, we can identify
the relative Loday construction $\L^C_{(X,Y)}(A, C; C)_\bullet$ 
with the Loday construction of the quotient:
\begin{prop} \label{prop:relative-quotient}
Let $A$ be an augmented commutative $C$-algebra. Then there is an
isomorphism of simplicial augmented commutative $C$-algebras
\begin{equation} \label{eq:quotient}
\L^C_{(X,Y)}(A, C; C)_\bullet \cong \L^C_{X/Y}(A;C)_\bullet
\end{equation}
where $X/Y$ has the equivalence class of $Y$ as a basepoint.
\end{prop}
\begin{proof}
We use Proposition \ref{prop:2.7} above and obtain that
$$ \L^C_{(X,Y)}(A, C; C)_\bullet  \cong \L^C_{X}(A; C)_\bullet
\wedge_{\L^C_{Y}(A; C)_\bullet} \L^C_{Y}(C; C)_\bullet$$
but $\L^C_{Y}(C; C)_\bullet$ is isomorphic to the constant simplicial
object $C_\bullet$ with $C$ in every simplicial degree. Thus
$$\L^C_{(X,Y)}(A, C; C)_\bullet  \cong \L^C_{X}(A; C)_\bullet
\wedge_{\L^C_{Y}(A; C)_\bullet} C_\bullet \cong \L^C_{X/Y}(A;
C)_\bullet \wedge_C C \cong \L^C_{X/Y}(A;
C)_\bullet  $$
as claimed.
\end{proof}

Proposition \ref{prop:relative-quotient} immediately gives rise to the
following spectral sequence.
\begin{prop}
If $C$ is a cofibrant commutative $S$-algebra and $A$ is a cofibrant commutative augmented $C$-algebra and if $Y$ is a pointed
simplicial subset of $X$,
then there is a spectral sequence
$$ E^2_{s,t} = \Tor_{s,t}^{\pi_*(\L^C_Y(A; C)_\bullet)}(\pi_*(\L^C_X(A; C)_\bullet),
 \pi_*C) \Rightarrow \pi_*(\L^C_{X/Y}(A; C)_\bullet).$$
\end{prop}
\begin{proof}
The isomorphism from Proposition \ref{prop:relative-quotient}
$$ \L^C_X(A; C)_\bullet \wedge_{\L^C_Y(A; C)_\bullet} C \cong \L^C_{X/Y}(A; C)_\bullet$$
induces a weak equivalence
$$ \L^C_X(A; C)_\bullet \wedge^L_{\L^C_Y(A; C)_\bullet} C \sim \L^C_{X/Y}(A; C)_\bullet$$
and we get the associated K\"unneth spectral sequence.
\end{proof}

\section{Spectral sequences with the relative Loday construction}
\label{sec:spsq}
In this section we set up some spectral sequences. Let $S$ be the
sphere spectrum and let $R$ be a
commutative $S$-algebra. Unadorned smash products will be over $S$.
We first recall some properties of the category of commutative $R$-algebras:
The category of commutative $R$-algebras is a topological model
category (\cite[VII.4.10]{EKMM}).
This implies that it is tensored over the category of unbased spaces
 and that for every sequence of  cofibrations $R \to A \to B$ of 
 commutative $S$-algebras and every relative
 CW-complex $(L,K)$ the map
 \[(A \otimes L) \wedge_{(A \otimes K)} (B \otimes K) \to B \otimes L  \]
 is a cofibration. For a  simplicial finite set  $X$ and and
 commutative $R$-algebra $A$ there is
 a natural isomorphism (see \cite[VII.3.2]{EKMM}):
 \begin{eqnarray*}
|\mathcal{L}^R_{X}(A)_\bullet| \cong A \otimes |X|.
\end{eqnarray*}
We define the Loday construction $\mathcal{L}^R_{|X|}(A)$ as  
$A \otimes |X|$. 
We get a similar definition in the pointed setting and we can define
the relative Loday construction for a pair of pointed CW complexes  
$Y \subset X$ and a sequence 
of maps of commutative $S$-algebras $R \ra A \ra B \ra C$  using  
Proposition \ref{prop:2.7} (a) as 
\begin{equation} \label{eq:defLodayspaces} 
\L^R_{(X, Y)}(A, B; C) := \L^R_X(A; C) \wedge^L_{\L^R_Y(A; C)} \L^R_Y(B; C). 
\end{equation}
\begin{thm}
Let $A$ be a cofibrant commutative $S$-algebra, and let
$B$ be a cofibrant commutative $A$-algebra. There is an 
equivalence of augmented commutative $B$-algebras
$$ \call_{(D^n, S^{n-1})}(A,B) \simeq \THH^{[n-1],A}(B)$$
for all $n$.
\end{thm}

\begin{proof}
We proceed by induction on $n$. For $n=1$,  $\call_{(D^1, S^{0})}(A,B)$ 
is the two-sided bar construction $B(B, A, B)$ which is a model for $B
\wedge^L_A B$. As we assumed $B$ to be a cofibrant commutative
$A$-algebra  $B \wedge^L_A B$ is weakly
equivalent to $B \wedge_A B$ which is $\THH^{[0],A}(B)$.
For the inductive step, we assume that
$\call_{(D^n,S^{n-1})}(A,B) \simeq \THH^{[n-1],A}(B)$. By \cite{V}, we
know that $\THH^{[n],A}(B)$ is weakly equivalent to the bar
construction $B^A(B,\THH^{[n-1],A}(B), B)$ by decomposing the
$n$-sphere into two hemispheres glued along an $(n-1)$-sphere.

We also know that $\call_{(D^{n+1}, S^n)}(A,B)$ can be built from two
half-disks of dimension $n+1$, part of whose boundary (the outside
edge) has $B$'s over it, and the other part (the $n$-disk we glue
along) has $A$'s over it. So by \eqref{eqn:2},

\begin{equation*}
\begin{split}
\call_{(D^{n+1}, S^n)}(A,B) &= \call_{(D^{n+1} \bigcup_{D^n} D^{n+1},
  D^n \bigcup_{S^{n-1}} D^n)}(A,B) \\
& \cong \call_{(D^{n+1},D^n)}(A,B) \wedge_{\call_{(D^n,
    S^{n-1})}(A,B)} \call_{(D^{n+1}, D^n)}(A,B).
\end{split}
\end{equation*}

For example, when $n=1$ we have

\[\call_{(D^{2}, S^1)}(A,B) \simeq \call_{\left(
\begin{adjustbox}{valign=c, scale=0.5}
\begin{tikzpicture}
\draw (0,0) -- (0,1);
\draw (0,1) arc (90:270:0.5);
\fill[opacity=0.4] (0,0) -- (0,1) arc (90:270:0.5);
\node[circle, minimum width=3pt, fill, inner sep=0pt,draw] at (0,0){};
\node[circle, minimum width=3pt, fill, inner sep=0pt,draw] at (0,1){};
\end{tikzpicture} \textbf{,} \, 
\end{adjustbox}
\begin{adjustbox}{valign=c,scale=0.5}
\begin{tikzpicture}
\draw (0,1) arc (90:270:0.5);
\node[circle, minimum width=3pt, fill, inner sep=0pt,draw] at (0,0){};
\node[circle, minimum width=3pt, fill, inner sep=0pt,draw] at (0,1){};
\end{tikzpicture}
\end{adjustbox}
\right)}(A,B) \wedge_{\call_{\left(
\begin{adjustbox}{valign=c,scale=0.5}
\begin{tikzpicture}
\draw (0,0) -- (0,1);
\end{tikzpicture} \textbf{,} \, 
\end{adjustbox}
\begin{adjustbox}{valign=c,scale=0.5}
\begin{tikzpicture}
\node[circle, minimum width=3pt, fill, inner sep=0pt,draw] at (0,0){};
\node[circle, minimum width=3pt, fill, inner sep=0pt,draw] at (0,1){};
\end{tikzpicture}
\end{adjustbox}
\right)}(A,B)} \call_{\left(
\begin{adjustbox}{valign=c,scale=0.5}
\begin{tikzpicture}
\draw (0,0) -- (0,1);
\draw (0,1) arc (90:-90:0.5);
\fill[opacity=0.4] (0,0) -- (0,1) arc (90:-90:0.5);
\node[circle, minimum width=3pt, fill, inner sep=0pt,draw] at (0,0){};
\node[circle, minimum width=3pt, fill, inner sep=0pt,draw] at (0,1){};
\end{tikzpicture} \textbf{,} \,
\end{adjustbox}
\begin{adjustbox}{valign=c,scale=0.5}
\begin{tikzpicture}
\draw (0,1) arc (90:-90:0.5);
\node[circle, minimum width=3pt, fill, inner sep=0pt,draw] at (0,0){};
\node[circle, minimum width=3pt, fill, inner sep=0pt,draw] at (0,1){};
\end{tikzpicture}
\end{adjustbox}
\right)}(A,B) \]

So we have
\begin{align*}
\call_{(D^{n+1},S^n)}(A,B)
	& \cong \call_{(D^{n+1}, D^n)}(A,B) \wedge_{\call_{(D^n,
            S^{n-1})}(A,B)} \call_{(D^{n+1},D^n)}(A,B)  	& 	\\
	& \simeq \call_{(\ast,\ast)}(A,B) \wedge^L_{\call_{(D^n,
            S^{n-1})}(A,B)} \call_{(\ast,\ast)}(A,B) 	&
        \text{(by homotopy invariance)}\\
	& \simeq B \wedge^L_{\call_{(D^n, S^{n-1})}(A,B)} B 		&	\\
	& \simeq B^A(B,\call_{(D^n, S^{n-1})}(A,B), B) 	&	\\
	& \simeq B^A(B, \THH^{[n-1],A}(B),B) 		&
        \text{(by assumption)}\\
	& \simeq \THH^{[n],A}(B).	&	\text{(by \cite{V})}
\end{align*}
\end{proof}

Let $C$ be a commutative $R$-algebra. Let $\mathcal{C}_{R/C}$ and
$\mathcal{C}_{C/C}$ denote the categories of commutative $R$-algebras
 over $C$ and of commutative $C$-algebras over $C$. Let $\mathcal{T}$
 denote the category of based spaces.
 We then have a functor
 \[\botimes_C \colon \mathcal{C}_{R/C} \times \mathcal{T} \to
 \mathcal{C}_{C/C}\]
 defined by $(A,X) \mapsto A \botimes_C X:= (A \otimes X) \wedge_A
 C$. Here, the map $A \to A \otimes X$ is given by the composition of
 the isomorphism $A \cong A \otimes {*}$ with the map induced by the
 inclusion of the basepoint.
 The augmentation $A \botimes_C X \to C$ is  given by
 \[(A \otimes X) \wedge_A C \to (A \otimes {*}) \otimes_A C \cong C.\]
 We have a natural homeomorphism
 \[\mathcal{C}_{C/C}(A \botimes_C X, B) \cong \mathcal{T}(X,
 \mathcal{C}_{R/C}(A, B)).\]
 Let $D \to E$ be a map in $\mathcal{C}_{R/C}$ such that the
 underlying map of commutative $R$-algebras is a cofibration. Let $K
 \to L$ be an inclusion of based spaces such that $(L,K)$ is a
 relative $CW$-complex. Then the natural map
 \[ (D \botimes_C L) \wedge_{(D \botimes_C K)} (E \botimes_C K) \to E
 \botimes_C L \]
 is a cofibration of commutative $R$-algebras. For $A \in
 \mathcal{C}_{R/C}$ and a simplicial  finite pointed set $X$, we have
 a natural isomorphism of $C$-algebras over $C$:
 \[ |\mathcal{L}^R_{(X,*)}(A; C)_\bullet| \cong A \botimes_C |X|.  \]
 \begin{thm} \label{thm:red-over}
 Let  $S \to A \to B \to C$ be a
 sequence of cofibrations of commutative $S$-algebras. Then
 \begin{enumerate}
  \item  $\THH^{[n], A}(B) \cong A \wedge_{\THH^{[n]}(A)}
    \THH^{[n]}(B)$ and
  \item  $\THH^{[n], A}(B; C ) \cong C \wedge_{\THH^{[n]}(A; C)}
    \THH^{[n]}(B;C)$.
 \end{enumerate}
In both cases, the smash product represents the derived smash
product. 
\end{thm}
\begin{proof}
 In order to show (a) we first prove that
 \[A \wedge_{\THH^{[n]}(A)} \THH^{[n]}(B)\]
 represents the derived smash product of $A$ and $\THH^{[n]}(B)$ over
 $\THH^{[n]}(A)$.
 For this we first show that $\THH^{[n]}(A)$ is a cofibrant
 commutative $S$-algebra: Since $A$ is a  cofibrant commutative
 $S$-algebra, it suffices to show that the unit $A \to \THH^{[n]}(A)$
 or equivalently that the map $A \otimes * \to A \otimes S^n$ is a
 cofibration of commutatative $S$-algebras. By the properties listed
 above
  the map
 \[(S \otimes S^n) \wedge_{(S \otimes *)} (A \otimes *) \to (A \otimes S^n)\]
 is a cofibration. The map $ S \otimes * \to S \otimes  S^n$ is an
 isomorphism because both sides can be identified with $S$.
 We get that
 \[(A \otimes *) \cong (S \otimes S^n) \wedge_{(S \otimes *)} (A
 \otimes *).\] Thus, $\THH^{[n]}(A)$ is cofibrant.
We now show that $\THH^{[n]}(A)  \to \THH^{[n]}(B)$ is a cofibration
of commutative $S$-algebras. For this it suffices to show that $A
\otimes S^n \to B \otimes S^n$ is a cofibration.
Since $A \to B$ is a cofibration, the map
\[ (A \otimes S^n) \otimes_{(A \otimes *)} (B \otimes *)  \to (B \otimes S^n)  \]
is a cofibration. The map $A \otimes * \to B \otimes *$ is a
cofibration because it can be identified with the map $A \to
B$. Because cofibrations are stable under cobase change the map
\[(A \otimes S^n) \to (A \otimes S^n) \otimes_{(A \otimes *)} (B
\otimes *) \] is a cofibration. Thus $A \otimes S^n \to B \otimes S^n$
is  a cofibration.
 Because  $\THH^{[n]}(A) \to \THH^{[n]}(B)$ is a cofibration between
 cofibrant commutative $S$-algebras,  we get by \cite[VII.7.4]{EKMM}
 that the functor
\[- \wedge_{\THH^{[n]}(A)} \THH^{[n]}(B) \]
preserves weak equivalences between cofibrant commutative
$S$-algebras. We factor the map $\THH^{[n]}(A) \to A$ as a cofibration
followed by an acyclic fibration
\[
\THH^{[n]}(A)   \rightarrowtail \tilde{A}  \stackrel{\sim}{\twoheadrightarrow} A
\]
and obtain a weak equivalence
\[\tilde{A} \wedge_{\THH^{[n]}(A)} \THH^{[n]}(B) \stackrel{\sim}{\to}
A \wedge_{\THH^{[n]}(A)} \THH^{[n]}(B).\]
By \cite[VII.6.7]{EKMM} the $S$-algebra
\[\tilde{A} \wedge_{\THH^{[n]}(A)} \THH^{[n]}(B)\]
represents the derived smash product of $A$ and $\THH^{[n]}(B)$ over
$\THH^{[n]}(A)$.

 We now show that there is an isomorphism of commutative $S$-algebras
 \[\THH^{[n], A}(B) \cong A \wedge_{\THH^{[n]}(A)} \THH^{[n]}(B).\]
 We start with the isomorphism of commutative  $S$-algebras
 \begin{eqnarray*}
  A \wedge_{\THH^{[n]}(A)} \THH^{[n]}(B) &\cong &
  |\mathcal{L}_{*}(A)_{\bullet}|
  \wedge_{|\mathcal{L}_{S^n}(A)_{\bullet}|}
  |\mathcal{L}_{S^n}(B)_{\bullet}| \\
  & \cong & |\mathcal{L}_{*}(A)_{\bullet}
  \wedge_{\mathcal{L}_{S^n}(A)_{\bullet}}
  \mathcal{L}_{S^n}(B)_{\bullet}|.
 \end{eqnarray*}
By a comparison of coequalizer diagrams we have,  for all $n$, isomorphisms of
commutative $S$-algebras:
 \[ A \wedge_{A^{\wedge n}} (B^{\wedge n}) \cong \underbrace{B \wedge_A \ldots
 \wedge_A B}_{n}\]
and these induce an isomorphism of simplicial
 commutative $S$-algebras
\[\mathcal{L}_{*}(A)_{\bullet}
\wedge_{\mathcal{L}_{S^n}(A)_{\bullet}}
\mathcal{L}_{S^n}(B)_{\bullet} \cong
\mathcal{L}^A_{S^n}(B)_{\bullet}.\]
This proves part (a) of the theorem.

 We now prove part (b). We again first show that
 \[  C \wedge_{\THH^{[n]}(A; C)} \THH^{[n]}(B;C)  \]
represents the derived smash product of $C$ and  $\THH^{[n]}(B;C)$
over $\THH^{[n]}(A; C)$.
For this it suffices to show that $\THH^{[n]}(A; C)$ is a cofibrant
commutative $S$-algebra and that the  map $\THH^{[n]}(A; C) \to
\THH^{[n]}(B; C)$ is a cofibration of commutative $S$-algebras.
The morphism $C \to \THH^{[n]}(A;C)$ is a cofibration because
\[C \to A \botimes_C S^n = (A \otimes S^n) \wedge_A C\]
 is a cofibration. Thus $\THH^{[n]}(A;C)$ is cofibrant.
The map $\THH^{[n]}(A;C) \to \THH^{[n]}(B;C)$ is a cofibration because
$A \botimes_C S^n \to B \botimes_C S^n$
can be written as
\[A \botimes_C S^n \to (A \botimes_C S^n) \wedge_{(A \botimes_C *)} (B
\botimes_C *) \rightarrowtail B \botimes_C S^n.\]
The first map of the composition is an isomorphism, because the map $A
\botimes_C * \to B \botimes_C  *$ identifies with the identity of $C$.

It remains to prove that there is an isomorphism of commutative $S$-algebras
\[  \THH^{[n], A}(B; C) \cong C \wedge_{\THH^{[n]}(A; C)} \THH^{[n]}(B;C). \]
This follows as above by using that we have an isomorphism of
commutative $S$-algebras
\[   C \wedge_{(A^{\wedge n} \wedge C)} (B^{\wedge n} \wedge C ) \cong
B^{\wedge_A n}  C\]
for all $n \geq 0$.
\end{proof}
\begin{rem}
The proof shows that Theorem \ref{thm:red-over} also holds for general
finite pointed simplicial sets $X$ and a sequence of cofibrations of
commutative $S$-algebras $S \to A \to B \to C$, giving us isomorphisms 
\begin{enumerate}
\item 
$\L^A_{|X|}(B) \cong A \wedge_{\L_{|X|}(A)} \L_{|X|}(B)$ and 
\item
$\L^A_{|X|}(B; C) \cong C \wedge_{\L_{|X|}(A; C)} \L_{|X|}(B; C)$. 
\end{enumerate}
\end{rem}
\begin{rem}
It is known that for
topological Hochschild homology, there is a difference between Galois
descent  and \'etale descent: John Rognes \cite{Rognes} 
developed the notion of Galois extensions for commutative $S$-algebras
and showed that for a Galois extension $A \ra B$ with finite Galois
group $G$ the canonical map $B \ra \THH^A(B)$ is a weak
equivalence \cite[Lemma 9.2.6]{Rognes}. Akhil Mathew \cite{Mathew}
provided an example of such a Galois
extension that does \emph{not} satisfy \'etale descent, \ie, the
pushout map
$$ B \wedge_A \THH(A) \ra \THH(B)$$
is not a weak equivalence. Theorem \ref{thm:red-over} doesn't
contradict this. We take a finite Galois extension $A \ra B$. Then we
obtain a weak equivalences
$$ B \ra \THH^A(B) \cong A \wedge_{\THH(A)} \THH(B).$$
But if we then smash this equivalence with $\THH(A)$ over $A$ the
resulting equivalence
\begin{equation} \label{eq:chain}
B \wedge_A \THH(A) \simeq (A \wedge_{\THH(A)} \THH(B)) \wedge_A
\THH(A)
\end{equation}
cannot be reduced to the statement that $B \wedge_A \THH(A)$ is
equivalent to $\THH(B)$: On the right hand side of
\eqref{eq:chain} we cannot reduce the $\THH(A)$-term because in the
smash product we use the augmentation map $\THH(A) \ra A$ and its
composite with the unit is not equivalent to the identity map.
\end{rem}

Let $R$ be a commutative $S$-algebra, and
$\mathcal{C}_R$  the category of commutative $R$-algebras.  Let $D$ be
the category $\{b \longleftarrow a 
\longrightarrow c\}$. Then the category ${}^D\mathcal{C}_R$
of functors from $D$ to $\mathcal{C}_R$ admits a
model category structure, where the weak equivalences
(resp. fibrations) are the maps that are objectwise weak equivalences
(resp. fibrations).
We have a cofibrant replacement functor ${}^D\mathcal{C}_R \to
{}^D\mathcal{C}_R$.
The homotopy pushout $B \twedge_A C$ of  a diagram $B
\longleftarrow A \longrightarrow C$ in $\mathcal{C}_R$
is constructed by taking the chosen  cofibrant replacement $B'
\longleftarrow A' \longrightarrow C'$  of the diagram and then taking
the usual pushout $B' \wedge_{A'} C'$.
One gets a functor
\[ (-) \twedge_{(-)}(-)\colon {}^D\mathcal{C}_R \to \mathcal{C}_R.\]
This functor sends  weak equivalences to weak equivalences. There is natural map
\[ B \twedge_A C \to B \wedge_A C\] which is a weak equivalence
when $A$ is cofibrant and $A \to B$ and $A \to C$
are cofibrations.
If $A$ is cofibrant, then $ B \twedge_A C$ is equivalent to the
derived smash product $B \wedge^L_A C$ of $B$ and $C$ over $A$.
One can show:
\begin{lem} \label{hp}
For a commutative diagram
\[\xymatrix{
E & D \ar[r] \ar[l] & F \\
B \ar[d] \ar[u] & A \ar[r] \ar[l] \ar[d] \ar[u] & C \ar[d] \ar[u] \\
H & G \ar[r] \ar[l] & I
}\]
in $\mathcal{C}_R$
there is a zig-zag of weak equivalences
\[  (E \twedge_D F)  \twedge_{(B \twedge_A C)}
 (H \twedge_G I) \sim (E \twedge_B H)
\twedge_{(D \twedge_A G)}  (F \twedge_CI) \]
over $(E \wedge_B H)  \wedge_{(D \wedge_A G)}  (F \wedge_C I)$
where
\[  (E \twedge_D F)  \twedge_{(B \twedge_A C)}
 (H \twedge_G I) \to  (E \wedge_B H)  \wedge_{(D \wedge_A
  G)}  (F \wedge_C I) \] is given by the composition of
the morphism
\[ (E \twedge_D F)   \twedge_{(B \twedge_A C)}
  (H \twedge_G I) \to   (E \wedge_D F)   \twedge_{(B
  \wedge_A C)}  (H \wedge_G I)
\to (E \wedge_D F)  \wedge_{(B \wedge_A C)}   (H \wedge_G I) \]
with the standard isomorphism
 \[ (E \wedge_D F)  {\wedge}_{(B \wedge_A C)}  (H \wedge_G I)
 \cong (E \wedge_B H)  \wedge_{(D \wedge_A G)}  (F \wedge_C I)  \]
and
\[(E \twedge_B H)  \twedge_{(D \twedge_A G)}
(F \twedge_CI)  \to (E \wedge_B H)   \wedge_{(D \wedge_A G)}
 (F \wedge_C I)\]
is given by
\[(E \twedge_B H)  \twedge_{(D \twedge_A G)}
(F \twedge_C I) \to (E \wedge_B H)   \twedge_{(D
  \wedge_A G)}  (F \wedge_C I)
 \to (E \wedge_B H)  \wedge_{(D \wedge_A G)}  (F \wedge_C I).\]

\end{lem}

\begin{thm} \label{thm:eva2nd}
Let $S \to A \to B \to C$ be a sequence of cofibrations of commutative
$S$-algebras. Then
\[\THH^{[n]}(B; C) \sim  \THH^{[n]}(A;C) \wedge^{L}_{\THH^{[n-1],
    A}(C)} \THH^{[n-1],B}(C).\]
\end{thm}
\begin{proof}
We work in the model category of commutative $S$-algebras.
For a map of commutative $S$-algebras $D \to E$ we define commutative
$S$-algebras $T^{[n],D}(E)$ augmented over $E$ inductively as follows:
Let $T^{[0],D}(E)$ be $E \twedge_D E$ and let $T^{[0],D}(E) \to E$ be defined by
\[E \twedge_D E \to E \wedge_D E\to E.\]
Set $T^{[n+1],D}(E) := E \twedge_{T^{[n],D}(E)} E $
and define $T^{[n+1],D}(E) \to E$ by
\[  E \twedge_{T^{[n],D}(E)} E \to E \wedge_{T^{[n],D}(E)} E \to E.\]
The $T^{[n],(-)}(E)$  are then endofunctors on the category of
commutative $S$-algebras over $E$. 
Using the decomposition $S^n = D^n \cup_{S^{n-1}} D^n$, one can show
that there are zig-zags of weak equivalences over $C$ (compare
\cite{V}) 
\begin{eqnarray*}
 \THH^{[n],A}(C) & \sim & C \twedge_{\THH^{[n-1],A}(C)} C \\
   \THH^{[n]}(A;C) & \sim & C \twedge_{\THH^{[n-1]}(A;C)} C. \\
\end{eqnarray*}
With that it follows that there are equivalences over $C$
\begin{eqnarray*}
 \THH^{[n],A}(C) & \sim &  T^{[n],A}(C) \\
 \THH^{[n]}(A;C) & \sim & T^{[n-1],C \wedge_S A}(C). \\
\end{eqnarray*}
The same is true for  $B$ instead of $A$.

It thus suffices to show:
\[T ^{[n], C \wedge_S B}(C) \sim T^{[n], C \wedge_S A}(C)
\twedge_{T^{[n], A}(C)}  T^{[n],B}(C).\]
We prove by induction on $n$ that these $S$-algebras are equivalent
via a zig-zag of weak equivalences over $C$ where the augmentation of
the right-hand side is given by
\[ T^{[n], C \wedge_S A}(C)  \twedge_{T^{[n], A}(C)}  T^{[n],B}(C)
\to  T^{[n], C \wedge_S A}(C)  {\wedge}_{T^{[n], A}(C)}
T^{[n],B}(C) \to C.   \]
We have an isomorphism $T^{[0], C \wedge_S B}(C) \cong T^{[0], (C
  \wedge_S A) \wedge_A B}(C)$  over $C$.
Because of the cofibrancy assumptions the map
\[(C \wedge_S A)  \twedge_A  B \xrightarrow{} (C \wedge_S A) \wedge_A B\]
is a weak equivalence. It induces a weak equivalence
\[( C \twedge_C C)  \twedge_{((C \wedge_S A) \twedge_A B)}  (C
\twedge_C C) \xrightarrow{\sim}
C  \twedge_{((C \wedge_S A) \wedge_A B)}  C = T^{[0], (C \wedge_S A) \wedge_A B}(C).\]
This is a map over $C$ if we endow the left-hand side with the augmentation
\[ ( C \twedge_C C)  \twedge_{((C \wedge_S A) \twedge_A B)}  (C
\twedge_C C) \to ( C \twedge_C C) \wedge_{((C \wedge_S A) \twedge_A
  B)} (C \twedge_C C) \to C  \twedge_C  C \to C. \]
By Lemma \ref{hp}, we have an equivalence
\[ ( C \twedge_C C)  \twedge_{((C \wedge_S A) \twedge_A B)}  (C
\twedge_C C) \sim (C \twedge_{(C \wedge_S A)} C) \twedge_{(C
  \twedge_A C)}  (C \twedge_B C)\]
and the right-hand side is equal to
  $T^{[0], C \wedge_S A}(C)  \twedge_{T^{[0],A}(C)}  T^{[0],
    B}(C)$.   The compatibility of the equivalence   with the isomorphism
  \[  ( C {\wedge}_C C)  {\wedge}_{((C \wedge_S A) {\wedge}_A B)}
  (C {\wedge}_C C) \cong (C {\wedge}_{(C \wedge_S A)} C)
  {\wedge}_{(C {\wedge}_A C)}  (C {\wedge}_B C)\]
  implies that it is an equivalence  over $C$.

  We now assume that the claim is true for $n$. Set $T' = T^{[n], C
    \wedge_S A}(C)  \twedge_{T^{[n], A}(C)}  T^{[n],B}(C)$. By
  induction hypothesis we have
  \[ T^{[n+1], C \wedge_S B}(C) = C  \twedge_{T^{[n],C \wedge_S
      B}(C)}  C \sim C \twedge_{T'}  C.\]
  via a zig-zag of weak equivalences over $C$.
 We have a weak equivalence
 \[(C \twedge_C C)  \twedge_{T'}  (C \twedge_C C) \xrightarrow{\sim}
 C   \twedge_{T'}  C.\]
 It is a map over $C$ if we endow the right-hand side with the augmentation
 \[(C \twedge_C C)  \twedge_{T'}   (C \twedge_C C) \to (C
 \twedge_C C) {\wedge}_{T'}   (C \twedge_C C) \to C  \twedge_C
   C \to C.  \]
 By Lemma \ref{hp} we have an equivalence
 \[  (C \twedge_C C)  \twedge_{T'}  (C \twedge_C C) \sim (C
 \twedge_{T^{[n], C \wedge_S A}(C)}  C) \twedge_{(C
   \twedge_{T^{[n],A}(C)} C)}  (C \twedge_{T^{[n],B}(C)} C)    \]
 and the right-hand side is equal to $T^{[n+1], C \wedge_S A}(C)
 \twedge_{T^{[n+1],A}(C)}  T^{[n+1], B}(C)$.
 Because of the compatibility with the isomorphism
 \begin{eqnarray*}
 & &  (C \wedge_C C)  {\wedge}_{(T^{[n], C\wedge_S A}(C)
   \wedge_{T^{[n], A}(C)} T^{[n],B}(C))}  (C {\wedge}_C C) \\
  & \cong &  (C {\wedge}_{T^{[n], C \wedge_S A}(C)} C)  {\wedge}_{(C
    {\wedge}_{T^{[n],A}(C)} C)}  (C {\wedge}_{T^{[n],B}(C)} C)
  \end{eqnarray*}
 it is an equivalence over $C$. This shows the induction step.
\end{proof}

\section{Applications}
\label{sec:ex}
\subsection{Thom spectra}
\begin{ex}
Christian Schlichtkrull \cite{schlichtkrull} gives a general formula
for the Loday construction on Thom spectra $\L_X(T(f); M)$ where
$f\colon A \rightarrow BF_{h\mathcal{I}}$ is
an $E_\infty$-map, $A$ is a grouplike $E_\infty$-space, and
$BF_{h\mathcal{I}}$ is a model for $BF=BGL_1(S)$, the classifying
space for stable spherical fibrations. The commutative $S$-algebra
$T(f)$ is the associated Thom spectrum for $f$ and $M$ is any
$T(f)$-module.

If we set $B=C$ in Theorem \ref{thm:red-over}, then we obtain
\begin{equation} \label{eq:juggling} \THH^A(B) \simeq B
  \wedge^L_{\THH^{[n]}(A; B)} \THH^{[n]}(B)
\end{equation}
so if there is a factorization as follows
$$\xymatrix{
A \ar[r]^f \ar[d]_h &  BF_{h\mathcal{I}} \\
B \ar[ru]_g
}$$
such that $h$ is a map of grouplike $E_\infty$-spaces, then we get an
induced map of commutative $S$-algebras $T(f) \ra T(g)$.

For $X$ a sphere and $M = T(g)$, we obtain \cite[Theorem 1]{schlichtkrull}
$$ \THH^{[n]}(T(f); T(g)) \simeq T(g) \wedge \Omega^\infty(S^n
\wedge \mathbb{A})_+$$
where $\mathbb{A}$ denotes the spectrum associated to $A$ such that
the map from $A$ to the underlying infinite loop space of
$\mathbb{A}$, $\Omega^\infty \mathbb{A}$, is a weak equivalence. 

Our juggling formula \eqref{eq:juggling} gives a formula for
higher $\THH$ of $T(g)$ as a commutative $T(f)$-algebra:
\begin{align*}
\THH^{[n], T(f)}(T(g)) & \simeq T(g) \wedge^L_{\THH^{[n]}(T(f);
  T(g))} \THH^{[n]}(T(g)) \\
& \simeq T(g) \wedge^L_{T(g) \wedge
  \Omega^\infty(S^n
\wedge \mathbb{A})_+} T(g) \wedge \Omega^\infty(S^n
\wedge \mathbb{B})_+.
\end{align*}
Important examples of such factorizations are listed for instance in
\cite[section
3]{Beardsley}. For example we can consider $BSU \ra BU$, $BU \ra BSO$
or $B\text{String} \ra B\text{Spin}$ to get $\THH^{[n], MSU}(MU)$,
$\THH^{[n], MU}(MSO)$ or $\THH^{[n],
  M\text{String}}(M\text{Spin})$. As these examples give rise to
Hopf-Galois extensions of ring spectra (see \cite{Rognes}) but not Galois
extensions, the above relative $\THH$-terms will be non-trivial.
\end{ex}

\subsection{$\THH^{[n], HA}(H\F_p)$ for  commutative pointed
  $\F_p$-monoid algebras $A$}
Hesselholt and Madsen \cite[Theorem 7.1]{HM} showed a splitting result
for topological Hochschild homology of pointed monoid rings. There is a
straightforward generalization of this splitting result to higher
order topological Hochschild homology in the commutative case.
Let $\Pi$ be a discrete pointed commutative monoid,
\ie, a commutative monoid in the category of based spaces with smash product.
Assume moreover that $\Pi$ is augmented, that is: admits a map of pointed
monoids to the pointed monoid $\{1, *\}$, where $1$ is the unit  and $*$
the base point.  As long as $1\neq *$ in $\Pi$, there always is such a
map: we can 
send all invertible elements in the monoid to $1$ and all the rest to
$*$.  In general, however, 
such an augmentation is not unique, so it needs to be part of the data.  We
consider the monoid algebra where the ground ring is the field
$\F_p$ and all unadorned tensor products are understood to be over
$\F_p$. Hesselholt-Madsen define
$\F_p[\Pi]$ as the $\F_p$ vector space with basis $\Pi$ modulo the
subspace generated by the basepoint.

The analogue of \cite[Theorem 7.1]{HM} is a splitting of augmented commutative 
$H\F_p$-algebras: 
\begin{equation} \label{eq:splitting}
\THH^{[n]}(H\F_p[\Pi]) \cong \THH^{[n]}(H\F_p) \wedge_{H\F_p}
\THH^{[n], H\F_p}(H\F_p[\Pi]). 
\end{equation}
Note that $\pi_*(\THH^{[n], H\F_p}(H\F_p[\Pi])) \cong
\HH_*^{\F_p, [n]}(\F_p[\Pi])$.

\begin{thm} \label{thm:Shuklamonoidalgs}
For any $\F_p$-algebra $\mathbb{F}_p[\Pi]$ on an augmented
commutative pointed monoid
$\Pi$ with $* \neq 1$ there is a 
weak equivalence 
$$ \THH^{[n], H\F_p[\Pi]}(H\F_p) \simeq H\F_p \wedge^L_{\THH^{[n],
    H\F_p}(H\F_p[\Pi])} H\F_p[\Pi]$$ 
of commutative augmented $H\F_p$-algebras. 
\end{thm}
\begin{proof}
Theorem \ref{thm:red-over} applied to a model of the augmentation map
$H\F_p[\Pi] \ra H\F_p$ that is a cofibration yields that
\begin{equation} \label{eq:thhhh}
\THH^{[n], H\F_p[\Pi]}(H\F_p)  \simeq \THH^{[n]}(H\F_p)
\wedge^L_{\THH^{[n]}(H\F_p[\Pi])} H\F_p[\Pi].
\end{equation}
We use the two-sided bar construction as model for the above derived
smash product and use   
the splittings in \eqref{eq:splitting} and \eqref{eq:thhhh} to obtain 
\begin{align*}
& B(\THH^{[n]}(H\F_p), \THH^{[n]}(H\F_p[\Pi]), H\F_p[\Pi]) \\ 
\simeq & B(\THH^{[n]}(H\F_p), \THH^{[n]}(H\F_p) \wedge_{H\F_p}
\THH^{[n], H\F_p}(H\F_p[\Pi]), H\F_p[\Pi]) \\ 
\simeq & B(\THH^{[n]}(H\F_p), \THH^{[n]}(H\F_p), H\F_p) \wedge_{H\F_p}
B(H\F_p, \THH^{[n], H\F_p}(H\F_p[\Pi]), H\F_p[\Pi]) \\  
\simeq & H\F_p \wedge_{H\F_p} B(H\F_p, \THH^{[n], H\F_p}(H\F_p[\Pi]), H\F_p[\Pi]) 
\end{align*}
which is a model of $H\F_p \wedge^L_{\THH^{[n], H\F_p}(H\F_p[\Pi])} H\F_p[\Pi]$. 
\end{proof}
We apply the above result to special cases of pointed commutative 
monoids, where we can
identify the necessary ingredients for the above result.
\begin{prop}
\begin{enumerate}
\item[]
\item
Consider the  polynomial algebra $\F_p[x]$ over $\F_p$ (with $|x|= 0$,
augmented by sending $x\mapsto 0$), and let 
$B_1'(x) = \F_p[x]$ and $B_{n+1}'(x) =
\Tor_{*,*}^{B_{n}'(x)}(\F_p, \F_p)$ with total grading.
Then
$$ \THH^{[n], H\F_p[x]}_*(H\F_p) \cong B_{n+2}'(x).$$
 \item 
Let $m$ be a natural number such that $p$ divides $m$, and let
$B_{1}''(m) = \Lambda_{\F_p}(\varepsilon x) \otimes
\Gamma_{\F_p}(\varphi^0 x)$ with $|\varepsilon x| = 1$ and $|\varphi^0
x| = 2$ and  $B_{n+1}''(m)= \Tor^{B_n''(m)}(\F_p,  \F_p)$. Then
$$ \THH^{[n], H\F_p[x]/x^m}_*(H\F_p) \cong B_{n+1}''(m).$$
\item 
Let $G$ be a finitely generated abelian group, so,
$G = \Z^m \oplus \bigoplus_{i=1}^N \Z/q_i^{\ell_i}$ for some primes
$q_i$. Then $\THH^{[n], H\F_p[G]}_*(H\F_p)$ can be expressed in terms of
a tensor product of  factors that are isomorphic to $\THH^{[n],
  H\F_p[x]}_*(H\F_p)$ or $\THH^{[n], H\F_p[x]/(x^{p^\ell})}_*(H\F_p)$
for some $\ell$.
\end{enumerate}
\end{prop}
\begin{proof}
We can rewrite 
$H\F_p \wedge^L_{\THH^{[n], H\F_p}(H\F_p[\Pi])} H\F_p[\Pi]$ as
$$ H\F_p \wedge_{H\F_p[\Pi]} H\F_p[\Pi]\wedge^L_{\THH^{[n],
    H\F_p}(H\F_p[\Pi])} H\F_p[\Pi]$$ 
which is equivalent to 
$$ H\F_p \wedge_{H\F_p[\Pi]} \THH^{[n+1], H\F_p}(H\F_p[\Pi]).$$
In \cite{BLPRZ} $\pi_*(\THH^{[n+1], H\F_p}(H\F_p[\Pi])) \cong
\HH_*^{[n+1], \F_p}(\F_p[\Pi])$ is  
calculated in the cases of the Proposition: 

For (a) we consider the pointed monoid $\Pi=\{0,1,x, x^2, \ldots\}$
whose associated 
pointed monoid ring $\F_p[\Pi]$ is the ring of polynomials over
$\F_p$. In  \cite[Theorem 8.6]{BLPRZ}, we show inductively that
$$ \HH^{[n], \F_p}_*(\F_p[x]) \cong \F_p[x] \otimes B_{n+1}'(x).$$ 
We also get inductively that the augmentation on $
\HH^{[n]}_*(\F_p[x]) $ is the identity on the $ \F_p[x]$ factor and
for degree reasons the obvious augmentation on $
B_{n+1}'(x)$. Therefore the claim follows.  

Higher Hochschild homology of truncated polynomial algebras of the form
$\F_p[x]/x^m$ for $m$ divisible by $p$ was calculated in \cite{BLPRZ}
(the case $m=p^\ell$) and \cite{BHLPRZ} (the general case).  The
result in those cases is 
$$ \HH^{[n]}_*(H\F_p[x]/x^m) \cong \F_p[x]/x^m \otimes B_n''(m)$$
where $B_1''(m) = \Lambda_{\F_p}(\varepsilon x) \otimes
\Gamma_{\F_p}(\varphi^0 x)$ with $|\varepsilon x| = 1$ and $|\varphi^0
x| = 2$ and where $B_{n+1}''(m)= \Tor^{B_n''(m)}(\F_p,  \F_p)$. This
implies (b). 

For a finitely generated abelian group as in (c) the group ring splits as
$$ \F_p[\Z]^{\otimes m} \otimes \bigotimes_{i=1}^N
\F_p[\Z/q_i^{\ell_i}].$$
The torsion groups with torsion prime to $p$ do not contribute to
higher (topological) Hochschild homology
because they are \'etale over
$\F_p$ (see \cite[Theorem 9.1]{BLPRZ}  and  \cite[Theorem
7.9]{Horel}). For the free factors we
use the fact that $\F_p[\Z]= \F_p[x^{\pm 1}]$ is \'etale over
$\F_p[x]$; for the factors with torsion that is a power of $p$, we use
the fact that $\F_p[\Z/p^\ell] \cong \F_p[y]/(y^{p^\ell} -1) \cong
\F_p[x]/(x^{p^\ell})$ by taking $x=y-1$. 
\end{proof}

\begin{rem}
Let $k$ be a commutative ring and let $A$ be a commutative
$k$-algebra.
In \cite{BHLPRZ} we define higher order Shukla homology of $A$ over
$k$ as
$$ \Shukla^{[n], k}_*(A) := \THH_*^{[n], Hk}(HA).$$
Thus the calculations above determine higher order Shukla homology for
commutative pointed monoid algebras over $\F_p$,  $\Shukla^{[n],
  \F_p[\Pi]}_*(\F_p)$. 
\end{rem}
\subsection{The examples $ko$, $ku$, $\ell$ and $\mathit{tmf}$} 

Angeltveit and Rognes calculate in \cite[5.13, 6.2]{AR} $H_*(\THH(ko); \F_2)$,
$H_*(\THH(\mathit{tmf}); \F_2)$, $H_*(\THH(ku); \F_2)$ and for any odd
prime $p$ they determine $H_*(\THH(\ell); \F_p)$ where $\ell \ra
ku_{(p)}$ is the Adams summand of $p$-local connective topological
complex K-theory.

The following lemma collects some immediate consequences of their
work, which were 
already noticed in \cite{G}.  These will be the basis of the
calculations in the results that follow the lemma. The index of a 
generator denotes its degree.
\begin{lem} \label{lem:thhwithcoeffs}
\begin{enumerate}
\item[]
\item
$$\THH_*(ko; \F_2) \cong \Lambda(x_5, x_7) \otimes \F_2[\mu_8].$$
\item
$$\THH_*(\mathit{tmf}; \F_2) \cong \Lambda(x_9, x_{13}, x_{15})
\otimes \F_2[\mu_{16}].  $$
\item
$$\THH_*(ku; \F_2) \cong \Lambda(x_3, x_7) \otimes \F_2[\mu_8].$$
\item
At any odd prime:
$$\THH_*(\ell; \F_p)  \cong \Lambda(x_{2p-1}, x_{2p^2-1}) \otimes \F_p[y_{2p^2}].$$
\end{enumerate}
\end{lem}
\begin{proof}
In all four cases Angeltveit and Rognes show that
$H_*(\THH(E); \F_p)$ is of the form $H_*(E; \F_p) \otimes A_E$ for
$A_E$ as described in  \eqref{AE} below, where 
$p=2$ for $E=ko, \mathit{tmf}, ku$ and $p$ is odd for $E = \ell$. We
rewrite  $\pi_*(\THH(E; H\F_p))$ as
\begin{align*}
\pi_*(\THH(E; H\F_p)) & \cong \pi_*(\THH(E) \wedge^L_E H\F_p) \\
                   & \cong \pi_*((\THH(E) \wedge H\F_p) \wedge^L_{E
                     \wedge H\F_p} H\F_p)
\end{align*}
and thus we get a spectral sequence
$$ E^2_{s,t} = \Tor_{s,t}^{H_*(E; \F_p)}(H_*(\THH(E); \F_p), \F_p)$$
converging to the homotopy groups of $\THH(E; H\F_p)$. As
$H_*(\THH(E); \F_p) \cong H_*(E; \F_p) \otimes A_E$ in all four cases,
the $E^2$-term above is concentrated in the $s=0$ column with
$$ E^2_{0,*} \cong A_E,$$
where
\begin{equation}\label{AE}
 A_E = \begin{cases}
\Lambda_{\F_2}(\sigma\bar{\xi}_1^4, \sigma\bar{\xi}_2^2) \otimes
\F_2[\sigma\bar{\xi}_3], & E = ko, \\[1ex]
\Lambda_{\F_2}(\sigma\bar{\xi}_1^8, \sigma\bar{\xi}_2^4,
\sigma\bar{\xi}_3^2) \otimes \F_2[\sigma\bar{\xi}_4], &
E=\mathit{tmf}, \\[1ex]
\Lambda_{\F_2}(\sigma\bar{\xi}_1^2, \sigma\bar{\xi}_2^2) \otimes
\F_2[\sigma\bar{\xi}_3], & E = ku, \text{  and }\\[1ex]
\Lambda_{\F_p}(\sigma\bar{\xi}_1, \sigma\bar{\xi}_2) \otimes
\F_p[\sigma \bar{\tau}_2], & E = \ell.
  \end{cases}
  \end{equation}
The degrees are the usual degrees in the dual of the Steenrod algebra,
hence at $2$ we have $|\xi_i|=2^i-1$ and at odd primes
$|\xi_i|=2p^i-2$ and  $|\tau_i|=2p^i-1$. We also have $|\sigma y| = |y|+1$. The
$\bar{(.)}$ denotes conjugation in the dual of the Steenrod
algebra. Counting degrees gives the claim.
\end{proof}
We can use the equivalence $\THH^A(B) \simeq B \wedge^L_{\THH(A; B)}
\THH(B)$ from Theorem \ref{thm:red-over} to deduce the following result.
\begin{thm} \label{thm:thh^E(hfp)}
There are additive isomorphisms
\begin{enumerate}
\item[]
\item
$\THH^{ko}(H\F_2) \cong \Gamma_{\F_2}(\rho^0x_5) \otimes
\Gamma_{\F_2}(\rho^0x_7) \otimes \F_2[\mu_2]/\mu_2^4$,
\item
$\THH^{\mathit{tmf}}(H\F_2) \cong \Gamma_{\F_2}(\rho^0x_9) \otimes
\Gamma_{\F_2}(\rho^0x_{13}) \otimes \Gamma_{\F_2}(\rho^0x_{15})
\otimes \F_2[\mu_2]/\mu_2^8$,
\item
$\THH^{ku}(H\F_2) \cong \Gamma_{\F_2}(\rho^0x_3) \otimes
\Gamma_{\F_2}(\rho^0x_7) \otimes \F_2[\mu_2]/\mu_2^4$,
\item
and for odd primes $p$ we get an additive isomorphism
$$\THH^{\ell}(H\F_p) \cong \Gamma_{\F_p}(\rho^0x_{2p-1}) \otimes
\Gamma_{\F_p}(\rho^0x_{2p^2-1}) \otimes \F_p[\mu_2]/\mu_2^{p^2}.$$
\end{enumerate}
Here $\rho^0$ raises degree by one.
\end{thm}
\begin{proof}
We use Theorem \ref{thm:red-over} in the case where $B=H\F_p$.
In \cite{B}, B\"okstedt shows that $\THH_*(H\F_p) \cong
\F_p[\mu_2]$ for all primes $p$. We give the details for the case
$ko$; the arguments for the other examples are completely
analogous.

The $E^2$-term of the spectral sequence is
\[\Tor_*^{\Lambda(x_5,x_7)\otimes \F_2[\mu_8]}(\F_2,\F_2[\mu_2]) \Longrightarrow
\THH_*^{ko}(H\F_2).\]
Since both $x_5$ and $x_7$ have odd degrees, they cannot act on $\mu_2$ other
than trivially.  Thus we can rewrite the left-hand side as
\[\Tor_*^{\Lambda(x_5,x_7)}(\F_2,\F_2) \otimes
\Tor_*^{\F_2[\mu_8]}(\F_2,\F_2[\mu_2]).\]

The explicit description of the generators in \cite[Theorem 6.2]{AR}
implies that the map $\F_2[\mu_8] \to \F_2[\mu_2]$ takes
$\mu_8$ to $\mu_2^4$, because $\mu_8$ corresponds to
$\sigma\bar{\xi}_3$ and Angeltveit and Rognes show \cite[proof of
5.12]{AR} that the
$\sigma\bar{\xi}_k$ satisfy
$$ (\sigma\bar{\xi}_k)^2 = \sigma\bar{\xi}_{k+1}$$
for $p=2$ and $\mu_2$ in B\"okstedt's calculation corresponds to
$\sigma\bar{\xi}_1$.

Therefore the right-hand $\Tor$ is isomorphic to
$\F_2[\mu_2]/\mu_2^4$.  Hence the $E^2$-term is isomorphic to
\[\Gamma_{\F_2}(\rho^0x_5)\otimes \Gamma_{\F_2}(\rho^0x_7)
\otimes \F_2[\mu_2]/\mu_2^4.\]

Since all the nonzero classes in this $E^2$-term have even total degree,
the spectral sequence must collapse at $E^2$.

In the case of the Adams summand, $\ell$, we work at odd primes and
here in \cite[proof of 5.12]{AR} the relation
$$ (\sigma\bar{\tau}_k)^p = \sigma\bar{\tau}_{k+1}$$
is shown. Hence $\sigma\bar{\tau}_2$ in $\THH_*(\ell; H\F_p)$
corresponds to $(\sigma\bar{\tau}_0)^{p^2}$ and $\sigma\bar{\tau}_0$ is
the element that represents $\mu_2$ at odd primes.
\end{proof}

\begin{rem}
In order to determine for instance $\THH^{ko}(H\F_2)$ multiplicatively
we would also need to control possible multiplicative extensions.  We
can show by degree considerations that
 $(\rho^kx_7)^2=0$, but not whether $(\rho^kx_5)^2$ must vanish.

As a general warning we discuss the case of the bar spectral sequence
in the case $\THH_*(ku; H\F_2)$. Here, we know the answer from Lemma
\ref{lem:thhwithcoeffs}:
$$ \THH_*(ku; H\F_2) \cong \Lambda(x_3, x_7) \otimes \F_2[\mu_8].$$
However, if we use the bar spectral sequence we get as the $E^2$-term
$$ \Tor^{\THH^{[0]}_*(ku; H\F_2)}_{*,*}(\F_2, \F_2) \cong
\Tor^{\F_2[\bar{\xi}_1^2,
\bar{\xi}_2^2, \bar{\xi}_k, k \geq 3]}_{*,*}(\F_2, \F_2)$$
because
$$ \THH^{[0]}_*(ku; H\F_2) \cong H_*(ku; \F_2) \cong \F_2[\bar{\xi}_1^2,
\bar{\xi}_2^2, \bar{\xi}_k, k \geq 3]$$
(see for instance \cite[6.1]{AR}). Hence the spectral sequence
collapses because
$$ E^2_{*,*} \cong \Lambda_{\F_2}(\sigma\bar{\xi}_1^2,
\sigma\bar{\xi}_2^2, \sigma\bar{\xi}_k, k \geq 3)$$
and all generators are concentrated on the $1$-line. But we know
that the exterior generators in $\Lambda_{\F_2}(\sigma\bar{\xi}_k, k
\geq 3)$ extend to form $\F_2[\mu_8]$, so there are highly non-trivial
multiplicative extensions in this spectral sequence.
\end{rem}

\begin{rem}
Veen established a Hopf-structure on the bar spectral sequence 
\cite[\S 7]{V} for higher order $\THH$ of $H\F_p$. His argument generalizes: 
The pinch maps $\mathbb{S}^n \ra  \mathbb{S}^n \vee \mathbb{S}^n$ give rise to 
a comultiplication 
$$ \THH^{[n]}(A; C) \ra \THH^{[n]}(A; C) \wedge_C \THH^{[n]}(A; C)$$
and as the multiplication on $\THH^{[n]}(A; C)$ is induced by the fold
map, both structures  
are compatible. For $\THH(A)$ this structure is heavily used in
\cite{AR}. 

If $A$ is connective, then we can consider $A \ra H(\pi_0A)$. For $C=
H\F_p$ this multiplication and 
comultiplication turns $\THH_*^{[n]}(A; H\F_p)$ into an $\F_p$-Hopf
algebra. Veen's arguments also transfer to yield that the bar spectral sequence 
$$ \Tor_{*,*}^{\THH_*^{[n]}(A; H\F_p)}(\F_p, \F_p) \Rightarrow
\THH_*^{[n+1]}(A; H\F_p)$$ 
is a spectral sequence of Hopf-algebras; in particular, the
differentials satisfy a Leibniz and a  
co-Leibniz rule and these facts let us determine the differentials in
certain cases.  
\end{rem}

\begin{thm} 
Additively,
\[\THH^{[2]}(ko; H\F_2) \cong \Gamma_{\F_2}(\rho^0x_5)\otimes
\Gamma_{\F_2}(\rho^0x_7)
\otimes \Lambda_{\F_2}(\varepsilon\mu_8).\]
Here, the degrees are $|\rho^0x_5|=6$, $|\rho^0x_7|=8$ and
$|\varepsilon\mu_8|=9$.
\end{thm}
\begin{proof}
Using the $\Tor$ spectral sequence we get
\[\Tor_{*,*}^{\Lambda(x_5)\otimes \Lambda(x_7) \otimes \F_2[\mu_8]}(\F_2,\F_2)
\Longrightarrow \THH_*^{[2]}(ko, H\F_2).\]
The $E^2$ page of the spectral sequence is of the form $\Gamma(\rho^0x_5)
\otimes \Gamma(\rho^0 x_7) \otimes \Lambda( \epsilon\mu_8)$.  This, in turn, is
isomorphic to
\[\bigotimes_{k=0}^{\infty}
\underbrace{\F_2(\rho^kx_5)/(\rho^kx_5)^2}_{\Lambda(\rho^kx_5)}
\otimes
\bigotimes_{\ell=0}^\infty
\underbrace{\F_2(\rho^kx_7)/(\rho^kx_7)^2}_{_{\Lambda(\rho^kx_7)}} \otimes
\Lambda (\epsilon \mu_8),\] with bidegrees $||\epsilon\mu_8|| =
(1,8)$, and $||\rho^kx_i|| =
(2^k,2^ki)$.  The claim is that the spectral sequence collapses
at $E^2$.

As the spectral sequence above is a spectral sequence of Hopf
algebras, the smallest nonzero differential must go from an
indecomposable element to a primitive
element.  As the only primitive element in $\Gamma(\rho^0x_i)$ is $\rho^0x_i$ we just
need to check that no differentials hit $\rho^0x_i$, $i=5, 7$, or
$\epsilon \mu_8$.  These
have bidegrees $(1,i)$ and $(1,8)$, respectively, and thus if they are hit by
$d^r$ of an indecomposable $d^r(\rho^k x_j)$, which would have
bidegree $(2^k-r, 2^k j +r-1)$,
we must have $2^k-r=1$ and $2^kj+r-1=5$, $7$, or $8$.  Substituting
$r=2^k -1$ into the
second expression, it becomes $2^k(j+1)-2$, where clearly $k\geq 1$
and $j=5$ or $7$, making $2^k(j+1)-2 \geq10$.  So the spectral
sequence collapses at
$E^2$.

\end{proof}

So far, we cannot rule out non-trivial multiplicative extensions.
These could be
nontrivial only if there is
a multiplicative generator  $\rho^kx_i$ or $\epsilon\mu_8$ with
whose square is not zero in $\THH^{[2]}_*(ko, H\F_2)$ (although it is
zero in the $E^\infty$ term).  This is possible only
if $(\rho^kx_i)^2$ has filtration degree less than $2^{k+1}$, or
$(\epsilon\mu_8)^2$ has filtration degree less
than $2$.  The latter is clearly impossible, since there is nothing in
total degree $18$ in filtration degrees
$0$ or $1$.  If $i=7$, by similar arguments we cannot have anything in
total degree
$2^{k+1}\cdot 8$ in filtration degree less than $2^{k+1}$: it would
have to be constructed out of elements
in bidegrees $(a, 5a)$, $(b, 7b)$ and possibly also one occurrence of
the element of bidegree $(1, 8)$, but then the first coordinates would
have to add up to at least $2^{k+1}$ to have the total degree equal to
$2^{k+1}\cdot 8$.
However, for the cases $(\rho^kx_5)^2$ one cannot rule out extensions
just for degree reasons.

\bigskip
\noindent
In a similar manner as above, we can exclude non-trivial differentials in the
other three cases:
\begin{prop} 
There are additive isomorphisms
\begin{enumerate}
\item[]
\item
$\THH_*^{[2]}(ku; H\F_2) \cong \Gamma_{\F_2}(\rho^0x_3) \otimes
\Gamma_{\F_2}(\rho^0x_7) \otimes \Lambda_{\F_2}(\varepsilon \mu_8)$,
\item
$\THH_*^{[2]}(\mathit{tmf}; H\F_2) \cong \Gamma_{\F_2}(\rho^0x_9) \otimes
\Gamma_{\F_2}(\rho^0x_{13}) \otimes
\Gamma_{\F_2}(\rho^0x_{15})\otimes \Lambda_{\F_2}(\varepsilon \mu_{16})$,
\item
and for any odd prime $p$ we get an additive isomorphism
$$\THH_*^{[2]}(\ell; H\F_p) \cong \Gamma_{\F_p}(\rho^0x_{2p-1}) \otimes
\Gamma_{\F_p}(\rho^0x_{2p^2-1}) \otimes \Lambda_{\F_p}(\varepsilon
\mu_{2p^2}).$$
\end{enumerate}
\end{prop}
\begin{proof}
\begin{enumerate}
\item[]
\item
In the case of $ku$ at the even prime we get a degree constraint for a
differential $d^r(\rho^kx_i)$ of the form
$$ (2^k-r, 2^ki + r -1) = (1, j)$$
where $j$ is $3, 7$ or $8$. Since $r\geq 2$ and $2^k-r=1$, we get
$k\geq 2$, but that would make the internal
degree at least $4(i+1)-2$ and this is bigger or equal to $14$,
hence doesn't occur.
\item
For $\mathit{tmf}$ the degree constraint is
$$ (2^k-r, 2^ki + r -1) = (1, j)$$
where $j$ is $9, 13, 15$ or $16$. Again $r\geq 2$ and $2^k-r=1$ imply
that $k\geq 2$, which makes  the internal degree at least $38$.
\item
For the Adams summand $\ell$ the degree condition is
$$ (p^k-r, p^ki + r -1) = (1, j)$$
where $j=2p-1, 2p^2-1$ or $2p^2$. As before, we get that $k \geq 2$
and therefore we get an internal degree of at least $2p^3-2$ which is
too big to be the degree of a primitive element.
\end{enumerate}
In all cases the differentials $d^r$, $r\geq 2$, all have to be
trivial and we get the result.
\end{proof}
\subsection{A splitting for $\THH^{[n], Hk}(HA)$ for
commutative  $k$-algebras $A$}
We apply Theorem \ref{thm:eva2nd} for a sequence of cofibrations of
commutative $S$-algebras of the form
$S \ra A \ra B = C$ and as
$\THH^{[n-1],B}(B) \simeq B$ we obtain a weak equivalence
\begin{equation} \label{eq:splitting2}
\THH^{[n]}(B) \simeq \THH^{[n]}(A; B) \wedge^L_{\THH^{[n-1], A}(B)}
B.
\end{equation}
In the special case of a sequence $S \ra Hk \ra HA = HA$ where $A$ is
a commutative $k$-algebra the formula in
\eqref{eq:splitting2} specializes to the following result.
\begin{prop}
For all commutative rings $k$ and all commutative $k$-algebras $A$ the
higher topological Hochschild homology of $HA$ splits as
$$ \THH^{[n]}(HA) \simeq \THH^{[n]}(Hk; HA) \wedge^L_{\Shukla^{[n-1],
    k}(A)} HA.$$
If $A$ is flat as a $k$-module, then higher Shukla homology reduces to
higher Hochschild homology and we obtain
$$ \THH^{[n]}(HA) \simeq \THH^{[n]}(Hk; HA) \wedge^L_{\HH^{[n-1],
    k}(A)} HA.$$
\end{prop}
In particular, this gives splitting results for number rings: For
$k=\Z$ and $A=\mathcal{O}_K$ a ring of integers in a number field we
get
$$  \THH^{[n]}(H\mathcal{O}_K) \simeq \THH^{[n]}(H\Z; H\mathcal{O}_K)
\wedge^L_{\HH^{[n-1],
    \Z}(\mathcal{O}_K)} H\mathcal{O}_K.$$
The (topological) Hochschild homology of $\mathcal{O}_K$ is known (see
\cite{LL,LM}).  However, the additive and multiplicative structure of these is
complicated enough that we cannot use the iteration methods of
\cite{BLPRZ, DLR}  and so we do not know
the higher order topological Hochschild homology of $\mathcal{O}_K$
with unreduced coefficients so far, nor its higher Shukla homology.

\begin{rem}
Beware that the splitting
\begin{align*}
\THH^{[n]}(HA) & \simeq \THH^{[n]}(Hk; HA) \wedge^L_{\Shukla^{[n-1],
    k}(A)} HA \\ 
& \simeq
(\THH^{[n]}(Hk) \wedge_{Hk}  HA) \wedge^L_{\Shukla^{[n-1], k}(A)} HA 
\end{align*}
cannot be rearranged to
$$  \THH^{[n]}(Hk) \wedge_{Hk} (HA \wedge^L_{\Shukla^{[n-1],
    k}(A)} HA) = \THH^{[n]}(Hk) \wedge_{Hk} \Shukla^{[n],
    k}(A)$$
because  the $\Shukla^{[n-1],k}(A)$-action on $ \THH^{[n]}(Hk; HA) $
does not usually factor through an action  on the coefficients $HA$.
If we could rearrange it that way, it would imply that
$\Shukla^{[n],k}(A)$ splits
off $\THH^{[n]}(HA)$, which is not true even for $n=1$: for example, if
we take $k=\Z$ and $A=\Z[i]$,  since $ \THH(H\Z)$ as the topological Hochschild
homology of a ring is equivalent to a product of Eilenberg Mac Lane spectra, which
B\"okstedt \cite{B} identified to be
$$ \THH(H\Z)\simeq H\Z \times\prod_{a=2}^\infty \Sigma ^{2a-1}H(\Z/a\Z),$$
then we get the formula
\begin{align*}
\pi_*( &  \THH(H\Z) \wedge_{H\Z} \Shukla(\Z[i])) \\
&\cong \pi_*(\Shukla(\Z[i])) \oplus\bigoplus_{a=2}^\infty  \pi_*(\Sigma ^{2a-1}H(\Z/a\Z) \wedge_{H\Z} \Shukla(\Z[i] ) ) \\
&\cong
\HH_*(\Z[i]) \oplus\bigoplus_{a=2}^\infty \Bigl( \HH_{*-2a+1}(\Z[i]) \otimes \Z/a\Z \oplus \Tor (\HH_{*-2a}(\Z[i]) ,\Z/a\Z )\Bigr)
\end{align*}
where  $ \HH_*(\Z[i]) =0$ when $*<0$. We also know that
$\HH_0(\Z[i])\cong \Z[i]$, $\HH_{2a-1}(\Z[i])\cong \Z[i]/2\Z[i]$, and
the positive even groups vanish.  Thus the number of copies of  $\Z/2\Z$'s in $\pi_n( \THH(H\Z) \wedge_{H\Z} \Shukla(\Z[i]))$ grows linearly with $n$.
On the other hand, by \cite{L},
$\THH_0(\Z[i])\cong\Z[i]$, $\THH_{2a-1}(\Z[i])\cong\Z[i]/2a\Z[i]$, and
the positive even groups vanish.

Such a splitting of $ \Shukla^{[n],k}(A)$
off $\THH^{[n]}(HA)$ \emph{does} hold under additional assumptions, for
instance in the case of commutative pointed monoid rings (see
\eqref{eq:splitting} above).
\end{rem}

\end{document}